\documentclass[12pt]{amsart}
\usepackage{amsmath,amsthm,amscd,amssymb}
\usepackage[margin=1.25in]{geometry}
\usepackage{color}

\theoremstyle{plain}

\theoremstyle{definition}

\theoremstyle{remark}

\newcommand{\R}{\mathbb{R}}

\newcommand{\ba}{\begin{aligned}}
\newcommand{\ea}{\end{aligned}}
\newcommand{\bt}{\begin{thm}}
\newcommand{\et}{\end{thm}}
\newcommand{\bc}{\begin{corollary}}
\newcommand{\ec}{\end{corollary}}
\newcommand{\bl}{\begin{lemma}}
\newcommand{\el}{\end{lemma}}
\newcommand{\bpf}{\begin{proof}}
\newcommand{\epf}{\end{proof}}
\newcommand{\bpb}{\begin{problem}}
\newcommand{\epb}{\end{problem}}
\newcommand{\bd}{\begin{definition}}
\newcommand{\ed}{\end{definition}}
\newcommand{\bn}{\begin{note}}
\newcommand{\en}{\end{note}}
\newcommand{\bp}{\begin{proposition}}
\newcommand{\ep}{\end{proposition}}
\newcommand{\be}{\begin{example}}
\newcommand{\ee}{\end{example}}
\newcommand{\bex}{\begin{exercise}}
\newcommand{\eex}{\end{exercise}}
\theoremstyle{plain}
\newtheorem{thm}{Theorem}[section]

\newtheorem{lemma}[thm]{Lemma}
\newtheorem{corollary}[thm]{Corollary}
\newtheorem{proposition}[thm]{Proposition}
\newtheorem{exercise}[thm]{Exercise}
\newtheorem{problem}[thm]{Problem}

\theoremstyle{definition}
\newtheorem{definition}[thm]{Definition}
\newtheorem{remark}[thm]{Remark}
\newtheorem{example}[thm]{Example}

\newtheorem {note}{Note}[section]
\swapnumbers
\theoremstyle{remark}

\theoremstyle{plain}

\title{Harnack Estimate for \\ the Endangered Species Equation}
\date{\today}

\begin{document}

\title[]{Harnack Estimate for \\ the Endangered Species Equation}

\author[]{Xiaodong Cao}
\address[]{ Department of Mathematics, Cornell University\\ Ithaca, NY 14853-4201, USA}
\email{cao@math.cornell.edu}
\date{\today}

\author[]{Mark Cerenzia}
\address[]{ Department of Mathematics, Cornell University\\ Ithaca, NY 14853-4201, USA}
\curraddr{Department of Operations Research and Financial Engineering, Princeton University, Princeton, NJ 08544}
\email{ cerenzia@princeton.edu}

\author[]{Demetre Kazaras}
\address[]{ Department of Mathematics, University of Oregon\\ Eugene, OR 97403-1222, USA}
\email{demetre@uoregon.edu}
\date{\today}



\maketitle

\begin{abstract}
We prove a differential Harnack inequality for  the \emph{Endangered Species Equation}, which is a nonlinear parabolic equation. Our derivation relies on an idea related to the parabolic maximum principle. As an application of this inequality, we will show that positive solutions to this equation must blowup in finite time.
\end{abstract}

\section{ Introduction}\label{intro}
We consider positive smooth solutions $f(x,t) : \R^n\times [0,\infty)\to \mathbb{R}$ to the following Cauchy problem:
\begin{equation} \label{gese}
\begin{cases}
\frac{\partial }{\partial t} f=\Delta f + f^p,\\
f(x,0)=f_0(x),
\end{cases}
\end{equation}
for $p>1$. It is well known that solutions for (\ref{gese}) may blow up in finite time. There are some early studies in \cite{kaplan63, friedman65, ito66} for certain quasi-linear parabolic equations that show finite time blow up under suitable conditions. H. Fujita \cite{fujita66} studied equation \eqref{gese} as an example of a general quasi-linear parabolic partial differential equation whose failure for long-time existence depends on both the spacial dimension $n$ and the power $p$,  but not on the (positive) initial value $f_0$, provided that $0<n(p-1)<2$. In \cite{hamilton11}, R. Hamilton  labeled \eqref{gese} as the  \emph{Endangered Species Equation (ESE)} since one may think of $f(x,t)$ as the population density of a certain species evolving in time t. The population evolves according to diffusion  (the term $\Delta f$), but the equation also incorporates an additional change (the term $f^p$) in population that results from a pair's meeting.  \\

Our main result Theorem \ref{maintheorem} is a differential Harnack inequality (also known as a Li-Yau type estimate). As an application, we find that \emph{any} positive solution with positive initial condition becomes unbounded in finite time. So the population becomes arbitrarily large no matter how endangered initially. Our approach thus provides a new derivation of Theorem 1 in Fujita \cite{fujita66} (our Theorem \ref{mainthm}). In addition, upon integrating this inequality along a space-time path, we can recover a classical Harnack inequality (Corollary \ref{cor.class}). 

The importance of parabolic Harnack inequalities, introduced in \cite{moser64, moser67}, is well-established.  Classical applications include deriving Holder continuity, obtaining Gaussian bounds for the heat kernel, and drawing many other conclusions about the underlying geometry of the space. While the study of differential Harnack inequalities and applications  originated in P. Li and S.-T. Yau \cite{ly86} (see also D. Aronson and P. B\'{e}nilan \cite{ab79} for a precursory form), this method was later brought into the study of geometric flows by Hamilton and played an important role  in the field, especially for the study of the Ricci flow (see for example, \cite{hamilton93}).   A more sophisticated use of such estimates for applications in geometry can be found in the details of the program for three dimensional geometrization.  This leads to G. Perelman's differential Harnack inequality  \cite[Corollary 9.3]{perelman1}, which is a crucial step in his solution to the Poincar\'e Conjecture. Since then, a systematic method to find Harnack inequality for geometric evolution equations was developed  in \cite{cao08, ch09}.\\

One of the main motivations in writing this paper is to suggest that the method developed in geometric flows can also be used for the study of  long-time existence (or non-existence)  for nonlinear parabolic equations, especially for finite-time blow-ups. In \cite{ab79}, the estimate was used to prove existence of solutions. For a probabilistic analysis of the Dirichlet problem associated to the endangered species equation with $1<p\leq 2$, see E.B. Dynkin \cite{dynkin92, dynkin93}. In his work, the solution $f$ of \eqref{gese} appears in the expression for the Laplace functional of a certain measure-valued Markov process, a so-called superprocess. Another goal of this paper is to generalize the estimate in \cite{hamilton11}, which also partially answers a question of Hamilton.\\

\paragraph{\textbf{Acknowledgements:}}  X. CaoÕs research was partially supported by NSF grant DMS 0904432. M. Cerenzia and D. Kazaras's research was supported by NSF grant DMS 0648208, through the Research Experience for Undergraduates Program at Cornell University. Note that any opinions, findings, and conclusions or recommendations expressed in this material are those of the authors and do not necessarily reflect the views of the National Science Foundation.

The authors would like to thank the referee and Dr. Hao Jia for their helpful comments and suggestions on an earlier version of this article. They would also like to thank Professor Robert Strichartz for his encouragement.
\section{Harnack Estimate}
In this section, we shall first derive our differential Harnack estimate. 
Let $f(x,t) \in C^{\infty}(\mathbb{R}^n\times[0,\infty))$ be a positive solution to \eqref{gese} and  $u:=\log f$, then we have
$$
u_t = \Delta u + \left \vert \nabla u \right \vert ^2 + e^{u(p-1)}.
$$
The main technical result, Theorem 2.2, relies upon calculation of the evolution of a Harnack quantity and use of  the parabolic maximum principle. The main object of our study is the following Harnack quantity
\begin{equation}\label{hq1}
H:= \alpha \Delta u + \beta \left \vert \nabla u \right \vert ^2 + c e^{u(p-1)} + \phi,
\end{equation}
where $\alpha, \beta, c\in\mathbb{R}$ and $\phi:\mathbb{R}^n\times[0,\infty)\to[0,\infty)$ will be chosen suitably later.  Our first task is to compute the evolution of $H$.

\begin{lemma}\label{mainlemma}
Suppose $f(x,t)$ is a positive solution to \eqref{gese}, $u=\log f$ and $H$ is defined as in \eqref{hq1}. Then we have 
\begin{align}\label{mainevolu}\nonumber 
H_t = & \Delta H + 2 \nabla H \cdot \nabla u + (p-1)e^{u(p-1)}H + 2(\alpha - \beta) \left \vert \nabla \nabla u \right \vert ^2 \\ \nonumber
& +\left ( \alpha (p-1) + \beta - cp  \right ) (p-1)e^{u(p-1)} \mid \nabla u \mid^2 \\
&- (p-1)e^{u(p-1)} \phi + \phi_t - \Delta  \phi - 2 \nabla  \phi \cdot \nabla u .
\end{align}
\end{lemma}

\begin{proof}The proof follows from direct calculation.  First note the evolution equations:
$$
\partial_t (\mid \nabla u \mid^2) = \Delta \mid \nabla u \mid^2 - 2\mid \nabla \nabla u \mid^2 + 2 \nabla \mid \nabla u \mid^2 \cdot \nabla u + 2 (p-1) e^{u(p-1)} \mid \nabla u \mid^2,
$$
and
$$
\partial_t (\Delta u) =  \Delta (\Delta u) + \Delta \mid \nabla u \mid^2 + (p-1) e^{u(p-1)} \Delta u + (p-1)^2 e^{u(p-1)} \mid \nabla u \mid^2.
$$ 
In the above, we use
\begin{equation}
\label{ricci2}
\Delta \mid \nabla u \mid^2 =  2\nabla u \cdot \nabla \Delta u +  2 \mid \nabla \nabla u \mid^2.
\end{equation}
Hence we  have
\begin{align}\nonumber 
H_t = &
\alpha \left [ \Delta (\Delta u) + \Delta\left \vert \nabla u \right \vert ^2  + (p-1) e^{u(p-1)} \Delta u + (p-1)^2 e^{u(p-1)} \left \vert \nabla u \right \vert ^2  \right ] \\ \nonumber
& + \beta \left [ \Delta \left \vert \nabla u \right \vert ^2  - 2\left \vert \nabla \nabla u \right \vert ^2  + 2 \nabla \left \vert \nabla u \right \vert ^2 \cdot \nabla u + 2 (p-1) e^{u(p-1)} \left \vert \nabla u \right \vert ^2  \right ] \\ \nonumber 
&+c(p-1)e^{u(p-1)}\left [ \Delta u + \left \vert \nabla u \right \vert ^2 + e^{u(p-1)}\right ]  +\phi_t .
\end{align}
Using (\ref{ricci2}) again, we arrive at
\begin{align}\nonumber 
H_t = &\Delta H + 2 \nabla H \cdot \nabla u
+ \alpha [  2|\nabla \nabla u | ^2  + (p-1) e^{u(p-1)} \Delta u  
+ (p-1)^2 e^{u(p-1)} \left \vert \nabla u \right \vert ^2  ]\\ \nonumber 
&+ \beta \left [2 (p-1) e^{u(p-1)}\left \vert \nabla u \right \vert ^2 - 2\left \vert \nabla \nabla u \right \vert ^2   \right ] 
+ c(p-1)e^{2u(p-1)}\\ \nonumber & - c(p-1)e^{u(p-1)}\left \vert \nabla u \right \vert ^2 - c(p-1)^2e^{u(p-1)}\left \vert \nabla u \right \vert ^2  + \phi_t - \Delta  \phi - 2 \nabla  \phi \cdot \nabla u.
\end{align}
The lemma then follows upon expanding and reordering. 
\end{proof}

\begin{thm}\label{maintheorem} 
Suppose $f(x,t)$ is a positive solution to $\eqref{gese}$ and  $u=\log f$.   If $\alpha, \beta, a$ and $c$  satisfy 
\begin{equation}\label{cond}
\alpha > \beta\geq 0, \ \ \frac{\alpha(p-1)+2\beta}{p}\geq c \geq \frac{(p-1)n\alpha^2}{4(\alpha-\beta)},
\end{equation} and
\begin{equation}\label{conda}
a\geq \frac{n\alpha^2}{2(\alpha-\beta)}>0,
\end{equation}
then we have
\begin{equation}
\label{harnack}
H_0\equiv \alpha \Delta u + \beta \left \vert \nabla u \right \vert ^2 + c e^{u(p-1)} + \frac{a}{t} \geq0
\end{equation}
 for all $t$. 
\end{thm}

\begin{remark}
Our conditions \eqref{cond} allow us to choose $\beta=0$ in \eqref{harnack}. However, while the argument in the proof below requires that $\beta>0$ initially, we can let $\beta \to 0$ at the end. \end{remark}

\begin{remark}
We will see that the inequality (\ref{condnp}) in the proof below exhibits a restriction on the power $p$ with respect to the spacial dimension $n$ similar to the primary condition of  Theorem 1 in Fujita \cite{fujita66}.
\end{remark}

\begin{remark}
This result also gives a partial answer to Question 4 in Hamilton \cite{hamilton11}.
\end{remark}

\begin{proof}
Choose $\phi_R(x,t)$ so that it is defined on the n-rectangle $R \subset \R^n$ made up of a cartesian product of $n$ intervals $[p_i, q_i]$ so that $\phi_R \rightarrow \infty$ if $x_i \rightarrow p_i, \ q_i $ or if $t \rightarrow 0$. More explicitly, we may take
\begin{equation}\label{defphi}
\phi_R(x,t) = \frac{a}{t} + \sum_{k=1}^n \left ( \frac{b}{(x_k - p_k)^2} + \frac{b}{(q_k - x_k)^2}    \right ),
\end{equation}
for $t>0$ and $x = (x_1, \ldots, x_n) \in R = \Pi_1^n [p_i, q_i]$, and extend it to be $\infty$ elsewhere. The corresponding Harnack quantity is
$$H_R= \alpha \Delta u + \beta \left \vert \nabla u \right \vert ^2 + c e^{u(p-1)} + \phi_R(x,t).$$
Note that $H_R \to H_0$ as the rectangle $R= \Pi_1^n [p_i, q_i]$  exhausts $\R^n$, and  $H_R>0$ for small $t$ since $\phi_R \to \infty$ as $t\to0$. 

For the sake of contradiction, assume that there exists a first time $t_0$ and point $x_0 \in R$ where $H_R(x_0, t_0) = 0$. At $(x_0,t_0)$, we have
$$
(H_R)_t \leq 0, \ \ \ \nabla H_R = 0, \ \ \ \Delta H_R \geq  0,
$$
and 
$$
\Delta u = -\frac{1}{\alpha} (\beta \left \vert \nabla u \right \vert ^2  + c e^{u(p-1)} + \phi_R).
$$
Applying Lemma 2.1 and Cauchy-Schwarz in the form $\left \vert \nabla \nabla u \right \vert ^2 \geq \frac{1}{n} (\Delta u)^2$ yields that
\begin{align} \label{expression}\nonumber
0  \geq &\frac{2(\alpha - \beta)}{n\alpha^2} [\beta \left \vert \nabla u \right \vert ^2  + c e^{u(p-1)} + \phi_R]^2   - (p-1)e^{u(p-1)} \phi_R \\
& + \left [ \alpha (p-1) + \beta - cp  \right ] (p-1)e^{u(p-1)} \mid \nabla u \mid^2 + (\phi_R)_t - \Delta  \phi_R - 2 \nabla  \phi_R \cdot \nabla u.
\end{align}
Set $X = e^{u(p-1)}$ and $Y= \left \vert \nabla u \right \vert^2 $. Expanding and combining terms gives
\begin{align}\label{eq6} \nonumber
0 \geq &\frac{2(\alpha-\beta)}{n \alpha^2}(c^2 X^2+\beta^2 Y^2) + \left [\alpha (p-1) - c p + \beta + \frac{4(\alpha-\beta)\beta c}{n \alpha^2 (p-1)} \right](p-1) XY  \\ 
 &+ \left [ \frac{4(\alpha-\beta)c}{n \alpha^2} - (p-1) \right]\phi_R X + \frac{4(\alpha-\beta)\beta }{n \alpha^2} \phi_R Y \\\nonumber
 &+ (\phi_R)_t  - \Delta  \phi_R - 2 \nabla  \phi_R \cdot \nabla u+ \frac{2(\alpha-\beta)}{n \alpha^2}\phi_R^2.
\end{align}
We now claim that the right hand side is in fact positive, which will give us a contradiction. First note the quadratic inequality
$$
\frac{4(\alpha-\beta) \beta }{n\alpha^2}  \phi_R Y - 2 \nabla  \phi_R \cdot \nabla u \geq \frac{-n \alpha^2 \mid  \nabla \phi_R \mid^2}{4 (\alpha-\beta) \beta \phi_R}.
$$
Given (\ref{cond}), it follows that
\begin{equation}\label{cond1}
\alpha (p-1) - c p + \beta + \frac{4(\alpha-\beta)\beta c}{n \alpha^2 (p-1)} \geq 0,\quad
\frac{4(\alpha-\beta)c}{n \alpha^2} - (p-1)\geq 0.
\end{equation}
Dropping several nonnegative terms in the right hand side of (\ref{eq6}), we  arrive at

$$
0 \geq (\phi_R)_t  - \Delta  \phi_R - \frac{n \alpha^2 \mid  \nabla \phi_R \mid^2}{4 (\alpha-\beta) \beta \phi_R} + \frac{2(\alpha-\beta)}{n \alpha^2}\phi_R^2.
$$
We then compute
$$
\Delta \phi_R = \sum_{i=1}^n \left ( \frac{6b}{(x_k-p_k)^4}+  \frac{6b}{(q_k-x_k)^4} \right ),
$$
$$
|\nabla \phi_R |^2 =  \sum_{k=1}^n \left( -  \frac{2b}{(x_k-p_k)^3}+\frac{2b}{(q_k-x_k)^3}  \right )^2,
$$
and observe that
\begin{align}\label{gradest}\nonumber
\frac{|\nabla \phi_R|^2}{\phi_R} & = \sum_{k=1}^n \left(  - \frac{2b}{(x_k-p_k)^3 \sqrt{\phi_R}}+\frac{2b}{(q_k-x_k)^3 \sqrt{\phi_R}}  \right )^2 \\
& \leq \sum_{k=1}^n \left(   \frac{2\sqrt{b}}{(x_k-p_k)^2 }+\frac{2\sqrt{b}}{(q_k-x_k)^2 }  \right )^2.
\end{align}
Set 
$$
 A :=  \frac{2(\alpha-\beta)}{n \alpha^2}>0 ,\ \ \ \ B:= \frac{n\alpha^2}{4(\alpha-\beta)\beta}>0.
$$ 
To arrive at a contradiction, it suffices to show that
\begin{equation}\label{eq12}
A \phi_R^2 - \Delta \phi_R - B \frac{ \mid \nabla \phi_R \mid^2}{\phi_R} + (\phi_R)_t>0.
\end{equation}
Now we choose $a$ as in (\ref{conda}), so that $Aa^2-a\geq 0$. Next, plugging (\ref{gradest}) and   (\ref{defphi}) into (\ref{eq12}), we conclude that it is sufficient to have  $b>0$ and 
$$
Ab^2-b(6+4B) >0 ,
$$
which reduces to 
$$
b>\frac{1}{A}(6+4B).
$$
In summary, the conditions on $a$ and $b$ are 
$$
a\geq  \frac{n\alpha^2}{2(\alpha-\beta)}, \ \ \ b> \frac{n\alpha^2}{2(\alpha-\beta)}\left[6+\frac{n \alpha^2} {(\alpha-\beta)\beta}\right].
$$
Recall that our constants $\alpha, \beta, c$ must satisfy $\alpha>\beta>0$ along with (\ref{cond1}):
$$
\alpha (p-1) - c p + \beta + \frac{4(\alpha-\beta)\beta c}{n \alpha^2 (p-1)} \geq 0,\ \ \ \  \frac{4(\alpha-\beta)c}{n \alpha^2} - (p-1) \geq 0.
$$
These latter two inequalities can be satisfied as long as we choose $c$ such that
$$
\frac{\alpha(p-1)+2\beta}{p}\geq c \geq \frac{(p-1)n\alpha^2}{4(\alpha-\beta)}.
$$
For given $n,p$, we may choose such $c$ as long as we have
\begin{equation}\label{condnp}
\frac{4(\alpha(p-1)+2\beta)(\alpha-\beta)}{\alpha^2} \geq p(p-1)n.
\end{equation}

Our choice of constants $\alpha, \beta, a, c$ now implies  that the right hand side of \eqref{eq6} is in fact positive, which is a contradiction.  Assuming the solution exists in all of space $\R^n$, we can let $R \to \R^n$ so that $\phi_R \to a/t$. 
This completes the proof of Theorem \ref{maintheorem}.
\end{proof}

\section{Applications}
In this section, we shall give a few applications of Theorem \ref{maintheorem}. We first recover and improve estimate in \cite{hamilton11}; then we use it to study long time existence problem of (\ref{gese}); finally we integrate along space-time curve to derive a classical Harnack inequality. 

\subsection{Hamilton's result}
In this subsection, we study the case of $n=1$ and $p=2$, which was studied  in  \cite{hamilton11} by Hamilton. In particular, we apply Theorem \ref{maintheorem} with $n=1$ and $p=2$  by picking $\alpha=1$, $\beta=0$, $c=\frac{1}{2}$, and $a=\frac{2}{3}$, to conclude
$$
\Delta u +\frac{1}{2}e^{u} +\frac{2}{3t} \geq 0.
$$
Recalling that $u=\log f$, we  recover \cite[Theorem 5.1]{hamilton11}: 

\begin{thm}[Hamilton \cite{hamilton11}] 
Let $f$ be a positive solution to \eqref{gese} with $n=1$, $p=2$. Then
\begin{equation}\label{hharnack}
f_t + \frac{2f}{3t} \geq \frac{f_x^2}{f} + \frac{f^2}{2}.
\end{equation}
\end{thm}
Furthermore, our proof in fact shows that \eqref{hharnack} can be improved by picking $a=\frac12$ and $c=\frac14$ to get 
$$
\Delta u +\frac{1}{4}e^{u} +\frac{1}{2t} \geq 0,
$$
yielding
\begin{equation}\label{hharnack2}
f_t + \frac{f}{2t} \geq \frac{f_x^2}{f} + \frac{3f^2}{4}.
\end{equation}

\begin{remark}In \cite{hamilton11}, Hamilton asks if it is possible to improve (\ref{hharnack}).
The above (\ref{hharnack2}) gives an affirmative answer to that question.
\end{remark}

If the dimension $n=2$, and $p=2$, a similar Harnack estimate can be derived  by picking $\alpha=1$, $\beta=0$, $a=1$ and $c=\frac12$, which we state as our next theorem.
\begin{thm}
Let $f$ be a positive solution to \eqref{gese} with $n=2$, $p=2$. Then
\begin{equation}\label{hharnack3}
f_t + \frac{f}{2t} \geq \frac{f_x^2}{f} + \frac{3f^2}{4}.
\end{equation}
\end{thm}

\subsection{Finite-time Blow Up}

In this subsection, we reprove Fujita's result \cite[Theorem 1]{fujita66}, which states that if $0<n(p-1)<2$, then any positive solution $f$ to \eqref{gese} will blow up in finite time, however small the positive initial value may be. We remark that Fujita also shows in the same paper that the condition $n(p-1)>2$ implies there exists some small positive initial data such that the solution exists for all time. For blow up in the case $n(p-1)=2$, see V. Galaktionov \cite{galak94}.\\

We in fact start by proving the following weaker version. Recall that $p>1$, so the lower bound for $n(p-1)$ in the next statement is always satisfied.

\bp 
Suppose that $f$ is a positive solution to $\eqref{gese}$, and $c$ is a constant satisfying $0<n(p-1)\leq c < 2$. Then $f$ blows up in finite time provided that 
\begin{equation} \label{lowbound}
f(x_0, t_0) \geq \left (\frac{4n}{2-c} \right )^{1/(p-1)}
\end{equation}
 at some point $(x_0, t_0)$.
\ep
\begin{proof}
Picking $\alpha=2$, $\beta=1$, $a=2n$ and $c$ such that $(p-1)n \leq c < 2$ in Theorem \ref{maintheorem} yields
$$
2 \Delta f - \frac{| \nabla f |^2}{f}  + c f^p +\frac{2n}{t}f \geq 0.
$$
Since $f_t=\Delta f+f^p$, we have
$$
2f_t - \frac{|\nabla f|^2}{f} + \frac{2nf}{t} \geq  (2-c)f^p.
$$
Hence
$$
2f_t + \frac{2nf}{t} \geq (2-c)f^p.
$$
The above inequality implies that
\begin{equation}\label{inverse}
2\frac{\partial}{\partial t} \left ( \frac{1}{f} \right ) \leq  \frac{1}{f} \left  ( \frac{2n}{t} - (2-c)f^{p-1}  \right ) =\frac{1}{f^{2-p}} \left  (\frac{2n}{t f^{p-1}} - (2-c)  \right ).
\end{equation}
Without loss of generality, we may assume that $f\geq (\frac{4n}{2-c})^{1/(p-1)}$ at the origin $x_0=0$ for $t_0=1$. This assumption together with (\ref{inverse}) gives
$$
2\frac{\partial}{\partial t} \left ( \frac{1}{f} \right )(0,t)\leq \frac{2-c}{f^{2-p}(0,t)} \left (\frac{1}{2t} - 1 \right )< 0,
$$
so that $f(0, t)$ is strictly increasing for $t\geq 1$ provided that $f(0, t)$ is finite. 

Now if $p>2$, then $f^{p-2}(0,t) \geq f^{p-2}(0,1)$ for $t \geq 1$ and (\ref{inverse}) becomes
$$
2\frac{\partial}{\partial t} \left ( \frac{1}{f} \right )(0,t) \leq \frac{2n}{tf(0,1)}-(2-c)f^{p-2}(0,1).
$$
On the other hand, if $1<p\leq 2$, manipulation of (\ref{inverse}) yields
\begin{equation*}\label{pl2}
\frac{2}{p-1} \frac{\partial }{\partial t} \left [ \left ( \frac{1}{f} \right )^{p-1} \right ](0,t)=2f^{2-p} \frac{\partial}{\partial t} \left ( \frac{1}{f} \right ) (0,t) \leq \frac{2n}{t f^{p-1}(0,1)} - (2-c) .
\end{equation*} 
In both cases, there exists  $\delta>0$, when $t$ is large enough, the right hand sides are less than $-\delta<0$,  so that $\frac{1}{f} \to 0$ in finite time. This proves our proposition.
\end{proof}

Intuitively, the above result says that if the population is ever large enough somewhere, then it will become unbounded in finite time. Through a parabolic rescaling argument, without loss of generality, we can always assume a given positive solution $f$ satisfies the condition \eqref{lowbound} of the previous proposition at some point. To see this, let $\lambda >0$ and $\delta \in \R$, define a new function $\tilde{f}(\tilde{x},\tilde{t}) := \lambda^\delta f(x, t)$, where $\tilde{x} := \lambda x$, $\tilde{t} := \lambda^2 t$. Then we can see that with the choice of $\delta := -\frac{2}{p-1}$, $\tilde{f}$ also satisfies equation \eqref{gese}: 
\begin{equation}
\label{rescale}
\frac{\partial}{\partial \tilde{t}}\tilde{f} = \lambda^{\delta-2} \left (\frac{\partial}{\partial t}f \right ) =\lambda^{\delta-2} (\Delta f)+ \frac{\lambda^{\delta-2}}{\lambda^{\delta p}} (\lambda^\delta f)^p = \tilde{\Delta} \tilde{f} + \tilde{f}^p.
\end{equation}
In particular, $\lambda>0$ is arbitrary and can be chosen so that condition \eqref{lowbound} is met by $\tilde{f}$. Once $\lambda$ is chosen and fixed,  $\tilde{f}$ remains bounded in finite time if and only if $f$ does. We have thus proved the following theorem, which says that no matter how small the initial population is, it will become unbounded in finite time.

\begin{thm}[Fujita \cite{fujita66}] \label{mainthm}
Let $0<n(p-1) < 2$. Then any positive solution $f$ to the equation $\eqref{gese}$ blows up in finite time.
\end{thm}

\subsection{Classical Harnack Inequality}
In this subsection,  we shall integrate our differential Harnack (\ref{harnack}) along a space-time path to derive  a classical Harnack type inequality, which provides a comparison of values 
of positive solutions at different points in space-time. 

\bc \label{cor.class}
Let $f$ be a positive solution to the generalized endangered species equation \eqref{gese} and $u = \log f$.  Let $\gamma(t) = (x(t),t)$, $t \in [t_1,t_2]$, be a space-time curve joining  two given points $(x_1,t_1), (x_2,t_2) \in \R^n \times [0, \infty)$ with $0 < t_1<t_2$. Assume further that $\alpha \geq 2 \beta$ so  $c \leq \alpha$, and  $a = \frac{n \alpha^2}{2(\alpha-\beta)}\leq n \alpha$. Then we have
\begin{equation}\label{intharnack}
f(x_1,t_1) \leq f(x_2,t_2) \left( \frac{t_2}{t_1} \right)^n \exp{\left[\frac{|x_2-x_1|^2}{2(t_2-t_1)}\right]}.
\end{equation}
\ec
\bpf
Recall the evolution equation for $u = \log f$ is
$$
u_t = \Delta u + | \nabla u |^2+e^{u(p-1)}.
$$
By our differential Harnack inequality (\ref{harnack}), we have $H\geq0$, this yields that 
$$
\begin{aligned}
\Delta u 
\geq \alpha^{-1} (- \beta | \nabla u |^2 - c e^{u(p-1)} - \frac{a}{t}).
\end{aligned}
$$
We then compute the evolution of $u$ along $\gamma$:
$$
\begin{aligned}
\frac{d}{d t} [u(x(t),t)] &=\nabla u \cdot \dot{x} + u_t  \\
& = \nabla u \cdot \dot{x} + \Delta u + | \nabla u |^2+e^{u(p-1)} \\
& \geq | \nabla u |^2 \left (1-\frac{\beta}{\alpha} \right ) + \nabla u \cdot \dot{x} - \frac{a}{\alpha t} + e^{u(p-1)}\left (1-\frac{c}{\alpha} \right ) \\
& \geq | \nabla u |^2 \left (\frac{1}{2}-\frac{\beta}{\alpha} \right ) - \frac{1}{2} | \dot{x}|^2 - \frac{a}{\alpha t} + e^{u(p-1)}\left (1-\frac{c}{\alpha} \right ) \\
& \geq  - \frac{1}{2} | \dot{x}|^2- \frac{a}{\alpha t}, \\
\end{aligned}
$$
where we have used the assumption $\alpha \geq 2 \beta$ with  $c \leq \alpha$ for the last inequality. Hence, we have
\begin{equation}\label{curve}
\frac{d}{d t} [-u(x(t),t)] \leq  \frac{1}{2} | \dot{x}|^2+ \frac{n}{t}.
\end{equation}
Integrating the above inequality \eqref{curve} along $\gamma$, and taking the infimum over all such space-time paths yields
$$
u(x_1,t_1) - u(x_2,t_2) \leq \inf_{\gamma(t) = (x(t),t)} \int_{t_1}^{t_2} \left[ \frac{1}{2} | \dot{x}|^2+ \frac{n}{ t} \right] .
$$
Recalling that $u = \log f$, we arrive at \eqref{intharnack}. 
\epf

\def\cprime{$'$}
\bibliographystyle{plain}


\end{document}